\documentclass[reqno,12pt]{amsart}    

\title
[]
{Evolution of spacelike surfaces in anti-De Sitter space by their Lagrangian angle}
\author[Knut Smoczyk
]{\sc Knut Smoczyk
}
\address{
Leibniz Universit\"at Hannover\newline
Institut f\"ur Differentialgeometrie\newline
Welfengarten 1\newline
30167 Hannover\newline
Germany}
\email{smoczyk@math.uni-hannover.de} 
\thanks{}
  
\subjclass[2010]{Primary 53C44; 
}   
\keywords{}
\date{July 6., 2011} 

\usepackage{amscd,amsfonts,mathrsfs,amsthm,enumerate}
\usepackage{amssymb, amsmath}          
\usepackage{stmaryrd}                          
\usepackage{epsfig}
\usepackage[shortalphabetic]{amsrefs}    
  
\def\real     #1{{\mathbb R^{#1}}}

\def\dd       #1#2#3{{#1}_{#2#3}}
\def\ddd      #1#2#3#4{{#1}_{#2#3#4}}
\def\dddd     #1#2#3#4#5{{#1}_{#2#3#4#5}}

\def\uu       #1#2#3{{#1}^{#2#3}}

\def\ud       #1#2#3{{#1}^{#2}_{\phantom{#2}{#3}}}
\def\du       #1#2#3{{#1}_{#2}^{\phantom{#2}{#3}}}

\def\uddd     #1#2#3#4#5{#1^#2_{\phantom{#2}#3#4#5}}

\def\hdd      #1#2#3#4{{#1}^{~^{\hspace{-5pt}(#2)}}_{{#3#4}}}

\def\dt       {\frac{d}{dt}\,}

\def\equationcolor {\color{black}}
\def\textcolor     {\color{black}}
\def\sca #1#2{\left\langle {#1},{#2}\right\rangle_{~_{\hspace{-5pt}2,4}}}
\def\scb #1#2{\left\langle {#1},{#2}\right\rangle_{~_{\hspace{-5pt}N}}}
\def\scc #1#2{\left\langle {#1},{#2}\right\rangle_{~_{\hspace{-5pt}1,3}}}
\def\ads {\operatorname{AdS}_3}
\def\hodge{\ast\hspace{-3pt}}

\def\bcoleq    {\begin{equation}\equationcolor}
\def\ecoleq    {\textcolor\end{equation}}

\def\bcoleqn   {\equationcolor\begin{eqnarray}}
\def\ecoleqn   {\end{eqnarray}\textcolor}

\newtheorem{theorem}{Theorem}[section]   
\newtheorem{lemma}[theorem]{Lemma}   
\newtheorem{corollary}[theorem]{Corollary}   
   
\newtheorem{remark}[theorem]{Remark}   
  
\theoremstyle{definition}   
\newtheorem{definition}[theorem]{Definition}

\newcommand{\bfig}{\begin{figure}}
\newcommand{\efig}{\end{figure}}

\begin{document}

\begin{abstract}
We study spacelike hypersurfaces in anti-De Sitter spacetime that evolve by the Lagrangian angle of their Gau\ss\ maps.
\end{abstract}
\maketitle
\section{Introduction}
In Lorentzian manifolds  spacelike hypersurfaces of prescribed
curvature (mean, scalar, Gau\ss\ etc.) are of great interest. To prove existence
of such hypersurfaces, many authors use either elliptic or parabolic methods. 
E.g. in \cite{Ecker1}, \cite{Ecker2}, \cite{Ecker-Huisken} mean curvature flow and variants of it have been used to generate spacelike hypersurfaces of prescribed mean curvature.

One advantage of the parabolic method 
is, that one can often prove existence of solutions merely under reasonable geometric assumptions
on the initial hypersurface - imposed  in terms of algebraic expressions involving the
second fundamental form - and e.g. the existence of barriers might not be needed.
In this paper we will see that there exists another interesting geometric flow for spacelike
hypersurfaces in anti-de Sitter spacetime, where the Gau\ss\ maps of the hypersurfaces move by the Lagrangian mean curvature flow. Our aim is to prove the following two theorems:

\begin{theorem}\label{theo main1}
Let $(N,g)$ be a time-oriented Lorentzian manifold of signature $(n,1)$ and of constant sectional curvature
$-\kappa<0$. Suppose $F_0:M\to N$ is a smooth spacelike immersion of a closed $n$-dimensional
manifold $M$.
There exists $T>0$ and a smooth solution $F:M\times[0,T)\to N$
of the evolution equation
\begin{eqnarray}
\dt F(p,t)&=&\phi(p,t)\nu(p,t)\,,\quad\forall p\in M,\,\forall t\in[0,T)\,,\label{eq bas1}\\
F(p,0)&=&F_0(p)\,,\quad\forall p\in M\,,\nonumber
\end{eqnarray}
where $\nu(p,t)$ denotes the future directed timelike unit normal at $F(p,t)$ and $\phi:M\to\real{}$
is the function
\begin{equation}\label{def f}
\phi=\frac{1}{\sqrt{\kappa}}\sum_{i=1}^n\arctan\frac{\lambda_i}{\sqrt{\kappa}}
\end{equation}
with $\lambda_1\le\dots\le\lambda_n$ denoting the principal curvatures of the immersion.
In particular, despite the fact that $\lambda=(\lambda_1,\dots,\lambda_n):M\to\real{n}$ is merely continuous,
the function $\phi$ is smooth.
\end{theorem}
\begin{theorem}\label{theo main2}
Under the same assumptions as in Theorem \ref{theo main1} assume in addition $n=2$ and that
the Gau\ss\ curvature $K=\lambda_1\lambda_2$ of the initial surface
satisfies
\begin{equation}\label{est 1}
|K|<\kappa\,.
\end{equation}
Then this condition is preserved during the flow, a smooth solution of (\ref{eq bas1}) exists for all $t\in[0,\infty)$ and $F$ converges smoothly and exponentially to a 
spacelike limit surface with vanishing mean curvature $H=\lambda_1+\lambda_2=0$ and with $-\kappa\le K<\kappa$ as $t\to \infty$. 
\end{theorem}
\begin{remark}
We remark that the condition $|K|<\kappa$ implies that the two components of the associated Gau\ss\ map $\mathscr{G}:M\to Gr_2^+(2,4)=\mathbb{H}^2_{1/\sqrt{2}}\times\mathbb{H}^2_{1/\sqrt{2}}$ (see below) are
immersions. Another interpretation can be given in terms of the 
two-positivity of the tensor $\dd Sij=\kappa\dd gij-\du hil\dd hlj$ w.r.t. the metric
$\dd\sigma ij:=\kappa\dd gij+\du hil\dd hlj$. where $\dd gij$ resp. $\dd hij$ denote the
first resp. second fundamental tensors of $F$. For some geometric evolution equations
two-positivity can be preserved, e.g. this has been shown in \cite{Tsui-Wang}.
\end{remark}
\begin{definition}
The function $\phi$ defined in equation (\ref{def f}) will be called the Lagrangian angle.
\end{definition}
The  last definition and also the flow defined in (\ref{eq bas1}) is motivated by the following observation:
Let us consider 
the anti-De Sitter space $\ads$ as the standard model of a Lorentzian space form with constant negative sectional curvature $-1$ represented by the hypersurface 
$$\ads=\{V\in\mathbb{R}^4_2:\sca VV=-1\}$$
and equipped with the induced Lorentzian metric, where
$\mathbb{R}^4_2$ denotes $\real{4}$ with its pseudo-Riemannian metric
$$\sca VW=V^1W^1+V^2W^2-V^3W^3-V^4W^4\,.$$
The Gau\ss\ map of a spacelike surface $M\subset\operatorname{AdS}_3\subset\mathbb{R}_2^4$
can be considered as a map $\mathscr{G}:M\to Gr_2^+(2,4)$, where $Gr_2^+(2,4)$ denotes
the Grassmannian of oriented spacelike planes in $\mathbb{R}_2^4$.
Moreover, up to scaling $Gr_2^+(2,4)$ is  isometric to $\mathbb{H}^2\times\mathbb{H}^2$.
By results of Torralbo and Urbano \cite{Torralbo}, \cite{Torralbo-Urbano}, the Gau\ss\ maps are Lagrangian.
In the appendix (Lemma \ref{lemm app1}) we will show that $\phi$ is the Lagrangian angle of  the Gau\ss\ map
and we will also prove (Lemma \ref{lemm app2}) that the Gau\ss\ maps of spacelike surfaces $M$ in $\operatorname{AdS}_3$
moving by (\ref{eq bas1}) will essentially, i.e. up to some tangential deformations, evolve by the  Lagrangian mean curvature flow. 

The flow defined by (\ref{eq bas1}) has been treated in some Riemannian cases, i.e. when $(N,g)$ is a Riemannian manifold. Andrews \cite{Andrews} studied the deformation of surfaces in $S^3$ by flows that allow an optimal control of
the Gau\ss\ curvature and he detected an optimal flow with the same driving term 
$\phi$ as defined in equation (\ref{def f}),
where in his case $\lambda_1,\lambda_2$ are the two principal curvatures of the surface $M\subset S^3$.
In the same paper the following was mentioned without proof:
If one considers $M\subset S^3\subset\real{4}$
as a submanifold of $\real{4}$, then the Gau\ss\ maps 
$\mathscr{G}:M\to Gr(2,4)$ of $M$ into the Grassmannian $Gr(2,4)$ of $2$-planes in $\real{4}$
will evolve by the mean curvature flow. We remark that a detailed analysis will actually show that this holds only up to tangential deformations
of the image in $Gr(2,4)$ (compare also with Lemma \ref{lemm app2} and with
the computations in the appendix). 
On the other hand, Castro and Urbano \cite{Castro-Urbano} proved
that the Gau\ss\ maps $\mathscr{G}:M\to Gr(2,4)$ of surfaces $M\subset S^3$ are Lagrangian.
Combining the results of Andrews and Castro, Urbano we see that an evolution of surfaces $M\subset S^3$ by the function
$\phi=\arctan\lambda_1+\arctan\lambda_2$ will induce (at least up to tangential deformations)
a Lagrangian mean curvature flow of their Gau\ss\ maps $\mathscr{G}:M\to Gr(2,4)$.
In analogy to Lemma \ref{lemm app1} one can also show that $\phi$ is the Lagrangian angle of  
the Gau\ss\ map, i.e. the mean curvature $1$-form $\tau$ of the Gau\ss\ map satisfies $\tau=d\phi$.

\noindent
In another case, if $\theta=du$ is an exact $1$-form on a flat manifold
$M$ and the graph of $du$ considered as a submanifold in the cotangent bundle $N:=T^*M$ (equipped with the
flat metric)
evolves by the Lagrangian mean curvature flow, then the potential $u$ evolves by
$$\frac{\partial}{\partial t} u=\sum_{i=1}^n\arctan \lambda_i\,,$$
where $\lambda_i$ are the eigenvalues of the Hessian $D^2u$ and $D$ denotes the flat connection.
In particular, $\phi=\sum_{i=1}^n\arctan \lambda_i$ is again the Lagrangian angle.  For details see 
\cite{Smoczyk-Wang} and \cite{Smoczyk}. Recently the Lagrangian mean curvature flow has
been generalized to the case of Lagrangian submanifolds in cotangent bundles $T^*M$
of Riemannian manifolds $(M,g)$. If the Lagrangian submanifold can be represented as the graph of a closed $1$-form $\theta\in\Omega^1(M)$, then there exists a generalized Lagrangian angle
similar to the function $\phi$ defined above, where now $\lambda_k$, $k=1,\dots, n$ are the eigenvalues of the symmetric tensor $D\theta$, $D$ denoting the Levi-Civita connection of the metric on $M$. 
For details see \cite{Smoczyk-Wang2}.

The organization of this paper is as follows: In section \ref{sec 2} we will introduce our notation
and briefly recall some of the geometry in Lorentzian manifolds of constant sectional curvature.
In section \ref{sec 3} we will first study arbitrary variations of spacelike hypersurfaces in
Lorentzian manifolds of constant sectional curvature $-\kappa<0$ and we will then characterize
the flow defined in (\ref{eq bas1}) as optimal w.r.t. the reaction term in the evolution
equation of the driving function (Lemma \ref{lemm opt}). In this section we will also prove the
smoothness of $\phi$ and Theorem \ref{theo main1}. Section 4 is completely devoted to the
two-dimensional case. We will first prove uniform $C^2$- and $C^1$-estimates and can then
establish the proof of Theorem \ref{theo main2}. In the appendix we will explain the relation
between $\phi$ and the Gau\ss\ map of spacelike surfaces in $\ads$ and we show that the Gau\ss\
maps will evolve under the Lagrangian mean curvature flow, if the spacelike surfaces in $\ads$ evolve by (\ref{eq bas1}).

\section{Geometry of spacelike hypersurfaces in Lorentzian space forms}\label{sec 2}
In this section we recall some basic facts concerning the geometry of
time-oriented {Lorentzian} manifolds $(N,g)$ of signature $(n,1)$ and of constant sectional curvature
$-\kappa<0$. If $\kappa=1$ and $N$ is complete, then $(N,g)$ is called a complete anti-De Sitter structure.
Of particular interest in this paper will be the three-dimensional case.
By results of Kulkarni, Raymond \cite{Kulkarni-Raymond} and Goldman \cite{Goldman} we know
that closed $3$-manifolds with a complete anti-De Sitter structure are necessarily
orientable Seifert fibre spaces with nonzero Euler number and with hyperbolic base.
In a celebrated paper by Mess \cite{Mess} (which despite its great influence it had, was unpublished
until recently; see also the ``Notes on Mess' paper" \cite{Mess-Notes}, published in the same volume), maximal globally
hyperbolic Cauchy-compact spacetimes (called ``domains of dependence") of constant curvature in $2+1$ dimensions were studied. Recall that a globally hyperbolic (Cauchy-compact) spacetime $N$ is a spacetime admitting a (compact) spacelike hypersurface $M$ such that every inextendable timelike curve intersects $M$ exactly once and such that the order relation is given by the existence of isometric embeddings.
Mess gave a classification in the flat and anti-De Sitter cases. The De Sitter case was also studied by Mess but a classification was obtained later by Scannell \cite{Scannell}.
In the anti-De Sitter case, domains of dependence are quotients of convex open sets of the anti-De Sitter space, by discrete groups of isometries. Mess exhibited a one-to-one correspondence between anti-De Sitter domains of dependence and pairs of points in the Teichm\"uller space. This result has been extended to $2+1$-dimensional anti-De Sitter domains of dependence having only a complete Cauchy surface by Barbot \cite{Barbot}
and Benedetti and Bonsante \cite{Benedetti-Bonsante}.
Locally, any Lorentz three-manifold with complete anti-De Sitter structure is isometric to
the classical model $\ads$ (or likewise to its simply connected
universal cover).

Suppose now that $F:M\to N$ is a smooth spacelike immersion of an $n$-dimensional oriented manifold $M$ into
a time-oriented spacetime of constant sectional curvature $-\kappa$, $\kappa>0$. To describe
the geometry of $(M, F^*g)$ and $(N,g)$ we will often use local coordinate systems $(U,x,\Omega)$ and $(V,y,\Lambda)$ 
for $M$ resp. $N$ where we assume here and in the following:
\begin {enumerate}[i)]
\item
$U\subset M$ is an open set around some point $p\in M$ and $x:U\to\Omega$ is a diffeomorphism between $U$ and some open set $\Omega\subset\real{n}$.
\item
$V\subset N$ is an open set around the point $q:=F(p)\in N$ and $y:V\to\Lambda$ is a diffeomorphism between $V$ and some open set $\Lambda\subset\real{n+1}$.
\item
The coordinate systems are always chosen in such a way that $F(U)\subset V$.
\item
For the set of coordinates on $M$ we will use Latin indices, i.e. $x=(x^i)_{i=1,\dots,n}$. Similarly, we will use Greek indices for the
coordinates on $N$, i.e. $y=(y^\alpha)_{\alpha=1,\dots,n+1}$.
\end{enumerate}
In these local coordinates, geometric quantities on $M$ and $N$ will then often be distinguished simply 
by use of Latin or Greek indices, e.g. $g=\dd g\alpha\beta dy^\alpha\otimes dy^\beta$ and $F^*g=\dd gijdx^i\otimes dx^j$ 
will denote  the Lorentzian resp. the induced Riemannian
metric tensors on $N$ resp. $M$. We also define $F^\alpha(x):=(y^\alpha\circ F)(x)$
and using the Einstein summation convention we have
$$\dd gij=\dd g\alpha\beta F^\alpha_iF^\beta_j\,,$$
where $F^\alpha_i:=\partial F^\alpha/\partial x^i$.
Let $\nabla$ denote the Levi-Civita connection on $(M,F^*g)$. The second fundamental tensor $A$ is by definition
$A=\nabla dF$, where 
$$dF=F^\alpha_i\frac{\partial}{\partial y^\alpha}\otimes dx^i\in\Gamma(F^*TN\otimes T^*M)$$ 
is the differential of $F$. The second fundamental tensor $h=\dd hij dx^i\otimes dx^j$ w.r.t. the future directed timelike unit normal 
$\nu=\nu^\alpha\frac{\partial}{\partial y^\alpha}$ along $F(M)$
is given by $h=-g(A,\nu)$ and can be expressed in local coordinates by Gau\ss' formula
\begin{equation}\label{eq gaussformula}
\nabla_i F^\alpha_j=A^\alpha_{ij}=\dd hij\nu^\alpha=\frac{\partial^2 F^\alpha}{\partial x^i\partial x^j}-\Gamma^k_{ij}\frac {\partial F^\alpha}{\partial x^k}-\Gamma^\alpha_{\beta\gamma}\frac{\partial F^\beta}{\partial x^i}\frac {\partial F^\gamma}{\partial x^j}\,,
\end{equation}
where $\Gamma^k_{ij}$ and $\Gamma^\alpha_{\beta\gamma}$ are the Christoffel symbols of $\dd gij$ resp. $\dd g\alpha\beta$.
The principal curvatures $\lambda_1,\dots,\lambda_n$
at a point $p\in M$ are the eigenvalues of the Weingarten map
$$\mathscr{W}:T_pM\to T_pM\,,\quad \mathscr{W}(V)=\nabla_V\nu$$
which in local coordinates applied to a vector $V=V^i\frac{\partial}{\partial x^i}$ is given by
$$\mathscr{W}V=\ud{\mathscr{W}}kiV^i\frac{\partial}{\partial x^k}\,,\quad\ud{\mathscr{W}}ki=\ud hki=\uu gkl\dd hli\,,$$
where $(\uu gkl)_{k,l=1,\dots,n}$ denotes the inverse of $(\dd gkl)_{k,l=1,\dots,n}$ and where indices will be raised and lowered w.r.t.
the metric tensors $\dd gij$ resp. $\uu gij$.

Two other important relations are given by the Codazzi equation
\begin{equation}\label{eq codazzi}
\nabla_i\dd hjk=\nabla_j\dd hik
\end{equation}
and the Gau\ss' equation which in view of the constancy of the sectional curvatures on $(N,g)$ is
\begin{equation}\label{eq gauss}
\dddd Rijkl=-\kappa(\dd gik\dd gjl-\dd gil\dd gjk)-\dd hik\dd hjl+\dd hil\dd hjk\,.
\end{equation}
Here $\dddd Rijkl$ is the Riemannian curvature tensor w.r.t. the metric $\dd gij$ on $M$.
Using the rule for interchanging derivatives together with (\ref{eq codazzi}) and (\ref{eq gauss}) we compute
\begin{eqnarray}
\nabla_i\nabla_j\dd hkl&=&\nabla_i\nabla_k\dd hlj\nonumber\\
&=&\nabla_k\nabla_i\dd hlj+\uddd Rmlki\dd hmj+\uddd Rmjki\dd hlm\nonumber\\
&=&\nabla_k\nabla_l\dd hij+\uddd Rmlki\dd hmj+\uddd Rmjki\dd hlm\,.\label{eq simons}
\end{eqnarray}
We will use (\ref{eq simons}) in the sequel (the trace  is known as Simons' identity).

\section{Variations of spacelike hypersurfaces in Lorentzian space forms}\label{sec 3}
In this section assume that for some $T>0$ we are given a smooth family of spacelike immersions
$$F:M\times[0,T)\to N$$
such that
\begin{equation}\label{eq evol1}
\dt F(p,t)=f(p,t)\nu(p,t)\,,\quad\forall p\in M,\,\forall\, t\in[0,T)\,,
\end{equation}
where $f(\cdot,t)$ is a smooth function depending smoothly on the principal curvatures $\lambda_1,\dots,\lambda_n$ of the
immersed hypersurface $M_t:=F(M,t)$, and where $\nu(p,t)$ is the future directed timelike unit normal at $F(p,t)$.

The evolution equations for the first and second fundamental forms are described in the next Lemma, a proof of which
can be found in the literature (c.f. \cite{Gerhardt}, see also \cite{Ecker-Huisken})
\begin{lemma}[Evolution equations]\label{evolution equations}
The evolution equations of the first and second fundamental form of spacelike hypersurfaces in Lorentzian manifolds
of constant sectional curvature $-\kappa$ evolving by (\ref{eq evol1}) are
\begin{eqnarray}
\dt\dd gij&=&2f\dd hij\,,\label{evst 2}\\
\dt\dd hij&=&\nabla_i\nabla_j f+f\du hik\dd hjk-\kappa f\dd gij\,.\label{evst 5}
\end{eqnarray}
\end{lemma}
For a smooth function $f$ depending on the eigenvalues $(\lambda_i)_{i=1,\dots,n}$
of the second fundamental form  let us define the tensors
$$\uu fij:=\frac{\partial f}{\partial \dd hij}\,,\quad f^{ij,kl}:=\frac{\partial^2f}{\partial\dd hkl\partial\dd hij}\,.$$
Then
\begin{eqnarray}
\nabla_i\nabla_j f&=&\nabla_i\left(\uu fkl\nabla_j\dd hkl\right)\nonumber\\
&=&\uu fkl\nabla_i\nabla_j\dd hkl+f^{kl,pq}\nabla_i\dd hpq\nabla_j\dd hkl\,.\nonumber
\end{eqnarray}
If we insert this into (\ref{evst 5}) and use (\ref{eq gauss}), (\ref{eq simons}), then we get
\begin{eqnarray}
\dt\dd hij&=&\uu fkl\nabla_k\nabla_l\dd hij+f^{kl,pq}\nabla_i\dd hpq\nabla_j\dd hkl\nonumber\\
&&-\kappa\uu fkl(\dd hjk\dd gil-\dd hij\dd gkl+\dd hkl\dd gij-\dd hil\dd gjk)\nonumber\\
&&+\uu fkl(-\du hjm\dd hmk\dd hil+\du him\dd hmj\dd hkl-\du hkm\dd hml\dd hij+\du  him\dd hml\dd hjk)\nonumber\\
&&+f\du hik\dd hjk-\kappa f\dd gij\,.\label{eq evol h2}
\end{eqnarray}
If $G$ is another smooth function that depends smoothly on the principal curvatures, we may consider $G$ as a function of the tensors $\dd hij$ and $\uu gij$, so that
for example 
\begin{eqnarray}
\nabla_jG
&=&\frac{\partial G}{\partial\dd hkl}\nabla_j\dd hkl+\frac{\partial G}{\partial\uu gkl}\nabla_j\uu gkl\nonumber\\
&=&\frac{\partial G}{\partial\dd hkl}\nabla_j\dd hkl\label{eq hypgrad}
\end{eqnarray}
and
\begin{eqnarray}
\nabla_i\nabla_jG&=&\frac{\partial^2G}{\partial x^i\partial x^j}-\Gamma^k_{ij}\frac{\partial G}{\partial x^k}\nonumber\\
&=&\frac{\partial G}{\partial\dd hkl}\nabla_i\nabla_j\dd hkl+\frac{\partial^2 G}{\partial\dd hkl\partial \dd hpq}\nabla_i\dd hkl\nabla_j\dd hpq\,.\label{eq hyphess}
\end{eqnarray}
Likewise
\begin{eqnarray}
\dt G&=&\frac{\partial G}{\partial\dd hij}\dt\dd hij+\frac{\partial G}{\partial\uu gij}\dt\uu gij\,.\label{eq hypG1}
\end{eqnarray}
Let us define
$$\uu Gkl:=\frac{\partial G}{\partial \dd hkl}\,,\quad G^{kl,pq}:=\frac{\partial^2 G}{\partial\dd hpq\partial\dd hkl}\,.$$
Then combining (\ref{evst 2}) and (\ref{evst 5})  one gets
\begin{lemma}\label{lemm g}
Under the flow (\ref{eq evol1}) the evolution equation of an arbitrary smooth function $G$,
depending smoothly on the eigenvalues $\lambda_1,\dots,\lambda_k$ of the Weingarten map
$\mathscr{W}$ is
\begin{eqnarray}
\dt G-\uu fij\nabla_i\nabla_j G&=&\uu Gij\uu fkl(\nabla_i\nabla_j\dd hkl-\nabla_k\nabla_l\dd hij)\nonumber\\
&&+\left(\uu Gij f^{pq,kl}-\uu fij G^{pq,kl}\right)\nabla_i\dd hkl\nabla_j\dd hpq\nonumber\\
&&+f\left(\uu Gij(\dd hil\ud hlj-\kappa \dd gij)-2\frac{\partial G}{\partial\uu gkl}\uu hkl\right)\,.\label{eq hypG2}
\end{eqnarray}
\end{lemma}

We observe that for $G=f$ the first two lines on the RHS vanish so that we obtain as a corollary:
\begin{corollary}
Under the flow (\ref{eq evol1}) the evolution equation of $f$ itself is
\begin{equation}\label{eq hypf}
\dt f=\uu fij\nabla_i\nabla_j f+f\left(\uu fij(\dd hil\ud hlj-\kappa \dd gij)-2\frac{\partial f}{\partial\uu gkl}\uu hkl\right)\,.
\end{equation}
\end{corollary}
So far, $f$ is arbitrary. If we want the flow to be parabolic, then we must assume that the tensor $\uu fij$
is positive definite. Our idea is to choose a function $f$ in such a way that the flow is parabolic and
such that the reaction terms on the RHS of (\ref{eq hypf}) simplify as much as reasonable,
e.g. so that the term in the brackets is constant. 

To make an ansatz, we first consider the case $n=1$, so that $f$ merely depends on the (mean) curvature $\lambda=\uu gkl\dd hkl$.
We then obtain
$$\uu fkl=f'\uu gkl\,,\quad\frac{\partial f}{\partial\uu gkl}=f'\dd hkl\,,$$
where $f'=\partial f/\partial \lambda$.
Hence
$$\uu fij(\dd hil\ud hlj-\kappa \dd gij)-2\frac{\partial f}{\partial\uu gkl}\uu hkl=-f'(\kappa+\lambda^2)\,.$$
The flow is parabolic if $f'>0$. Thus we are looking for a monotone increasing (in $\lambda$)
function $f$ for which
$$f'(\kappa+\lambda^2)=c$$
for some constant $c$. 
But $\kappa>0$, $f'>0$ and $\kappa+\lambda^2>0$ imply $c>0$ and then 
$$f(\lambda)=\frac{c}{\sqrt{\kappa}}\arctan\frac{\lambda}{\sqrt{\kappa}}+a\,,$$
with some arbitrary constants $a$ and $c>0$. We will choose $c=1$ and $a=0$.
For general $n$ we now simply sum over all
eigenvalues of the Weingarten map and thus define
\begin{eqnarray}
f(p):=\phi(p)=
\frac{1}{\sqrt{\kappa}}\sum_{j=1}^n\arctan\frac{\lambda_j(p)}{\sqrt{\kappa}}\,.\label{eq deff}
\end{eqnarray}

%

For each $k$ the function $\lambda_k$ is continuous but in general not smooth. Surprisingly, as the next Lemma shows, the function $\phi$ in (\ref{eq deff}) is smooth.
\begin{lemma}\label{lemm anglesmooth}
Suppose $F:M\to (N,g)$ is a smooth spacelike immersion into a time-oriented 
Lorentzian manifold of constant sectional
curvature $-\kappa<0$ and let $\lambda=(\lambda_1,\dots,\lambda_n)$ denote the principal curvature functions on $M$. The function $\phi:M\to\real{}$ defined in (\ref{eq deff}) is smooth.
\end{lemma}
\begin{proof}
Let $\lambda_1\le\lambda_2\le\dots\le\lambda_n:M\to\real{}$ be the principal curvature functions on $M$.
It is well-known that each $\lambda_k$ is continuous but in general not smooth. 
Hence $\phi$ is at least continuous. W.l.o.g. we may assume $\kappa=1$.
On $\real{n}$ we consider the following three smooth functions:
\begin{eqnarray}
&&\alpha:\real{n}\to\left(-\frac{n\pi}{2},\frac{n\pi}{2}\right)\,,\quad\alpha(x_1,\dots,x_n)=\sum_{k=1}^n\arctan x_k\,,\nonumber\\
&&a:\real{n}\to\real{}\,,\quad a(x_1,\dots,x_n):=\sum_{k=0}^{\left[\frac{n-1}{2}\right]}(-1)^{k}s_{2k+1}(x_1,\dots,x_n)\,,\nonumber\\
&&b:\real{n}\to\real{}\,,\quad b(x_1,\dots,x_n):=\sum_{k=0}^{\left[\frac{n}{2}\right]}
(-1)^{k}s_{2k}(x_1,\dots,x_n)\,,\nonumber
\end{eqnarray}
where $s_k$ are the elementary symmetric functions, e.g. 
$$s_0(x_1,\dots,x_n)=1\,,\quad s_1(x_1,\dots,x_n)=\sum_{k=1}^nx_k$$
and
$$s_n(x_1,\dots,x_n)=\prod_{k=1}^nx_k\,.$$
By induction one can show
$$a^2(x_1,\dots,x_n)+b^2(x_1,\dots,x_n)=\prod_{k=1}^n(1+x_k^2)\ge 1\,,$$
so that $a$ and $b$ cannot vanish simultaneously in a point $x\in\real{n}$.
Let 
$U_a:=\{x\in\real{n}:a(x)\neq 0\}$ and $U_b:=\{x\in\real{n}:b(x)\neq 0\}$. Then $U_a\cup U_b=\real{n}$ and both $U_a$ and $U_b$ are open. We obtain two smooth functions
\begin{eqnarray}
\alpha_a:U_a\to\left(-\frac{\pi}{2},\frac{\pi}{2}\right)\,,\quad\alpha_a(x):=-\arctan\frac{b(x)}{a(x)}\,,\nonumber\\
\alpha_b:U_b\to\left(-\frac{\pi}{2},\frac{\pi}{2}\right)\,,\quad\alpha_b(x):=\arctan\frac{a(x)}{b(x)}\,.\nonumber
\end{eqnarray}
Now at each $x\in U_a$ we have
$$\frac{\partial \alpha_a}{\partial x_k}=\frac{1}{1+x_k^2}=\frac{\partial\alpha}{\partial x_k}\,,$$
and likewise at any $x\in U_b$
$$\frac{\partial \alpha_b}{\partial x_k}=\frac{1}{1+x_k^2}=\frac{\partial\alpha}{\partial x_k}\,,$$
so that $\alpha_a-\alpha$ resp. $\alpha_b-\alpha$ are constant on each connected component of $U_a$ resp. $U_b$. We are now ready to prove the smoothness of $\phi$. Let $p\in M$ be arbitrary. At $p$ we must
either have $a(\lambda_1(p),\dots,\lambda_n(p))\neq 0$ or $b(\lambda_1(p),\dots,\lambda_n(p))\neq 0$.
W.l.o.g. assume $b(\lambda(p))\neq 0$ (the other case can be treated similarly). Since the 
elementary symmetric functions
$$\tilde s_k:M\to\real{}\,, \quad p\mapsto s_k(\lambda_1(p),\dots,\lambda_n(p))$$
are smooth ($\tilde s_0=1$, $\tilde s_1=H,\dots,\tilde s_n=K$), we know that $\tilde a:=a\circ\lambda$ and 
$\tilde b:=b\circ\lambda$ are smooth functions on
all of $M$ since they can be expressed in terms of the elementary symmetric functions. Choose a small open neighborhood $U\subset M$ around $p$ with $b(\lambda(q))\neq 0$ for all
$q\in U$ and such that $\lambda(q)$ lies in the same connected component of $U_b$ for any $q\in U$ (the latter works due to the continuity of $\lambda$). Then $\phi=\alpha\circ\lambda$ implies
that $\phi-\arctan\frac{\tilde a}{\tilde b}$ is a constant function on $U$. Since $\arctan\frac{\tilde a}{\tilde b}$ is smooth
on $U$, so must be $\phi$. This proves the claim.\\
\end{proof}

Let us define
\begin{equation}\label{defi sigma}
\dd\sigma ij:=\kappa\dd gij+\du hil\dd hlj\,.
\end{equation}
From the construction of $\phi$ we get
\begin{eqnarray}
\frac{\partial \phi}{\partial\dd hkl}&=&\uu\sigma kl\,,\label{eq fpart1}\\
\frac{\partial \phi}{\partial\uu gkl}&=&\dd gik\dd hjl\uu \sigma ij\,,\label{eq fpart2}\\
\phi^{pq,kl}=\frac{\partial \uu\sigma pq}{\partial\dd hkl}
&=&-(\uu\sigma pk\uu\sigma qj+\uu\sigma qk\uu\sigma pj)\ud hlj\,.\label{eq fpart4}
\end{eqnarray}
where $\uu\sigma ij$ shall denote the inverse of $\dd\sigma ij$.
\begin{eqnarray}
\end{eqnarray}

{\bf Proof of Theorem \ref{theo main1}:}
By Lemma \ref{lemm anglesmooth} the function $\phi$ is smooth and since $\uu\sigma ij=\frac{\partial \phi}{\partial\dd hij}$ is positive definite, 
the flow defined by (\ref{eq bas1}) is parabolic. The statement now follows from the standard
theory of parabolic evolution equations on smooth compact manifolds.

\hfill$\square$

Applying (\ref{eq fpart1}) and (\ref{eq fpart2}) to the general evolution equation
(\ref{eq hypf}) in case $f=\phi$ we obtain
\begin{lemma}\label{lemm opt}
Under the flow given by equation (\ref{eq bas1}) we have
\begin{equation}\label{eq evol f3}
\dt\phi=\uu\sigma ij\nabla_i\nabla_j\phi-n\phi\,.
\end{equation}
\end{lemma}
\begin{proof}
This follows directly from the construction of $\phi$ and likewise from
equations (\ref{eq fpart1}), (\ref{eq fpart2}):
\begin{eqnarray}
\uu fij(\dd hil\ud hlj-\kappa \dd gij)-2\frac{\partial f}{\partial\uu gkl}\uu hkl
&=&\uu\sigma ij(\dd hil\ud hlj-\kappa \dd gij)-2\dd gik\dd hjl\uu \sigma ij\uu hkl\nonumber\\
&=&-\uu\sigma ij(\kappa\dd gij+\dd hil\ud hlj)\nonumber\\
&=&-\uu\sigma ij\dd \sigma ij=-n\,.\nonumber
\end{eqnarray}
\end{proof}
A direct consequence of (\ref{eq evol f3}) and the maximum principle gives:
\begin{lemma}\label{est phi}
Under the flow given by equation (\ref{eq bas1}) the Lagrangian angle $\phi$ satisfies the  estimate:
\begin{equation}\nonumber
\inf_{q\in M}\phi(q,0)\le\phi(p,t)e^{nt}\le\sup_{q\in M}\phi(q,0)\,,\quad\forall (p,t)\in M\times[0,T)\,.
\end{equation}
\end{lemma}

Note that the quantity $\dd gik\dd hjl\uu \sigma ij$ is symmetric in $k$ and $l$. More
generally, for any non-negative integer $r$ we define a tensor $\hdd hris$ by
$$\hdd hris:=\begin{cases}
\dd gis&, r=0\\
\dd his&, r=1\\
\hdd h{r-1}ij\ud hjs&, r\ge 2\,.
\end{cases}$$
Then we have
\begin{lemma}
For all integers $r,s\ge 0$ the following symmetry holds:
\begin{equation}\label{eq hddsymm}
\uu\sigma kl\hdd hrki\hdd hslj=\uu\sigma kl\hdd hrkj\hdd hsli\,.
\end{equation}
\end{lemma}
\begin{proof}
We choose an orthonormal basis $e_1,\dots,e_n$ at a point $p\in M$ so that $\dd hij$ becomes diagonal at $p$, i.e.
$$\dd hij=\operatorname{diag}(\lambda_1,\dots,\lambda_n)\,.$$
Then all tensors $\hdd hkij$ become diagonal at $p$ as well, more precisely
$$\hdd hkij=\operatorname{diag}(\lambda_1^k,\dots,\lambda_n^k)\,.$$
In addition we have at $p$
$$\dd\sigma ij=\operatorname{diag}(\kappa+\lambda_1^2,\dots,\kappa+\lambda_n^2)\,,\quad\uu\sigma ij=\operatorname{diag}\left(\frac{1}{\kappa+\lambda_1^2},
\dots,\frac{1}{\kappa+\lambda_n^2}\right).$$
This implies the symmetries.
\end{proof}

\section{The two-dimensional case}\label{sec 4}
In the two-dimensional case we are able to prove a longtime existence and convergence result under the assumption
that the Gau\ss\ curvature $K$ of the spacelike surface is strictly bounded by
$$-\kappa<K<\kappa.$$
Therefore, in this section let us assume $n=2$ and let $K=\lambda_1\lambda_2$ be the Gau\ss\ curvature and $H=\lambda_1+\lambda_2$ the
mean curvature. 

From
$K=\frac{1}{2}\left(H^2-|h|^2\right)$, where $|h|^2=\lambda_1^2+\lambda_2^2$
is the squared norm of the second fundamental form, and from $n=2$ we get
\begin{equation}\label{eq gg1}
\hdd h2ij=H\dd hij-K\dd gij\,.
\end{equation}
An easy computation yields
\begin{eqnarray}
\uu\sigma ij
&=&\frac{1}{\kappa H^2+(\kappa-K)^2}\left((2\kappa+|h|^2)\uu gij-\kappa\uu gij-\uu gik\uu gjl\hdd h2kl\right)\nonumber\\
&=&\frac{1}{\kappa H^2+(\kappa-K)^2}\left((2(\kappa-K)+H^2)\uu gij-\kappa\uu gij-H\uu hij+K\uu gij\right)\nonumber\\
&=&\frac{1}{\kappa H^2+(\kappa-K)^2}\left((\kappa-K+H^2)\uu gij-H\uu hij\right)\,.\label{eq gg6}
\end{eqnarray}
Moreover we compute
\begin{equation}\label{eq gg2}
\frac{\partial H}{\partial\dd hij}=\uu gij\,,\quad\frac{\partial K}{\partial\dd hij}=H\uu gij-\uu hij
\end{equation}
and
\begin{equation}\label{eq gg3}
\frac{\partial H}{\partial\uu gij}=\dd hij\,,\quad\frac{\partial K}{\partial\uu gij}=H\dd hij-\hdd h2ij=K\dd gij\,.
\end{equation}
Like in the previous sections, let $G$ denote an arbitrary function depending smoothly on the eigenvalues $\lambda_1,\lambda_2$ of the Weingarten map. Since $\lambda_1$ and $\lambda_2$ can be computed from $H$ and $K$, we
may assume that $G$ depends only on $H$ and $K$. We set
$$G_H:=\frac{\partial G}{\partial H}\,,\quad G_K:=\frac{\partial G}{\partial K}\,.$$
From (\ref{eq gg2}) and (\ref{eq gg3}) we conclude
\begin{eqnarray}
\uu Gij=\frac{\partial G}{\partial \dd hij}&=&G_H\frac{\partial H}{\partial\dd hij}+G_K\frac{\partial K}{\partial\dd hij}\nonumber\\
&=&G_H\uu gij+G_K(H\uu gij-\uu hij)\nonumber\\
&=&(G_H+HG_K)\uu gij-G_K\uu hij\label{eq gg4}
\end{eqnarray}
and
\begin{eqnarray}
\frac{\partial G}{\partial\uu gij}&=&KG_K\dd gij+G_H\dd hij\,.\label{eq gg5}
\end{eqnarray}
From the general evolution equation of $G$ given by equation (\ref{eq hypG2}) we now derive
\begin{eqnarray}
\dt G-\uu \sigma ij\nabla_i\nabla_j G&=&\left(\uu Gij \sigma^{pq,kl}-\uu \sigma ij G^{pq,kl}\right)\nabla_i\dd hkl\nabla_j\dd hpq\nonumber\\
&&+\uu Gij\uu \sigma kl(\nabla_i\nabla_j\dd hkl-\nabla_k\nabla_l\dd hij)\nonumber\\
&&+f\left(\uu Gij(\hdd h2ij- \kappa\dd gij)-2\frac{\partial G}{\partial\uu gkl}\uu hkl\right)\,.\nonumber
\end{eqnarray}
We first simplify the last term
\begin{eqnarray}
&&\uu Gij(\hdd h2ij- \kappa\dd gij)-2\frac{\partial G}{\partial\uu gkl}\uu hkl
\nonumber\\
&=&\bigl((G_H+HG_K)\uu gij-G_K\uu hij\bigr)\bigl(H\dd hij-(\kappa+K)\dd gij\bigr)
-2(KG_K\dd gij+G_H\dd hij)\uu hij\nonumber\\
&=&(G_H+HG_K)(H^2-2(\kappa+K))-HG_K|h|^2+H(\kappa+K)G_K-2HKG_K-2|h|^2G_H\nonumber\\
&=&G_H(H^2-2(\kappa+K)-2|h|^2)+HG_K(H^2-2(\kappa+K)-|h|^2+\kappa+K-2K)\nonumber\\
&=&-(H^2+2(\kappa-K))G_H-H(\kappa+K)G_K
\end{eqnarray}

Let us also compute the second term:
\begin{eqnarray}
&&\uu Gij\uu \sigma kl(\nabla_i\nabla_j\dd hkl-\nabla_k\nabla_l\dd hij)\nonumber\\
&=&-\uu Gij\uu \sigma kl(\dddd Rmlik\ud hmj+\dddd Rmjik\ud hml)\nonumber\\
&=&(\kappa+K)\uu Gij\uu\sigma kl(\dd hij\dd gkl-\dd hkj\dd gli+\dd  hil\dd gjk-\dd hkl\dd gij)\nonumber\\
&=&(\kappa+K)(\uu Gij\dd hij\uu\sigma kl\dd gkl-\uu Gij\dd gij\uu\sigma kl\dd hkl)\,.\nonumber
\end{eqnarray}
On the other hand we compute
\begin{eqnarray}
\uu Gij\dd hij&=&\bigl[(G_H+HG_K)\uu gij-G_K\uu hij\bigr]\dd hij\nonumber\\
&=&HG_H+2KG_K\,,\nonumber
\end{eqnarray}
\begin{eqnarray}
\uu Gij\dd gij&=&\bigl[(G_H+HG_K)\uu gij-G_K\uu hij\bigr]\dd gij\nonumber\\
&=&2G_H+HG_K\,,\nonumber
\end{eqnarray}
\begin{eqnarray}
\uu\sigma kl\dd hkl&=&\frac{1}{\kappa H^2+(\kappa-K)^2}\left((\kappa-K+H^2)\uu gkl-H\uu hkl\right)\dd hkl\nonumber\\
&=&\frac{H(\kappa-K+H^2-|h|^2)}{\kappa H^2+(\kappa-K)^2}\nonumber\\
&=&\frac{H(\kappa+K)}{\kappa H^2+(\kappa-K)^2}\nonumber
\end{eqnarray}
and
\begin{eqnarray}
\uu\sigma kl\dd gkl&=&\frac{1}{\kappa H^2+(\kappa-K)^2}\left((\kappa-K+H^2)\uu gkl-H\uu hkl\right)\dd gkl\nonumber\\
&=&\frac{2(\kappa-K+H^2)-H^2}{\kappa H^2+(\kappa -K)^2}\nonumber\\
&=&\frac{2(\kappa-K)+H^2}{\kappa H^2+(\kappa -K)^2}\nonumber
\end{eqnarray}
so that
\begin{eqnarray}
&&\uu Gij\uu \sigma kl(\nabla_i\nabla_j\dd hkl-\nabla_k\nabla_l\dd hij)\nonumber\\
&=&(\kappa+K)(\uu Gij\dd hij\uu\sigma kl\dd gkl-\uu Gij\dd gij\uu\sigma kl\dd hkl)\nonumber\\
&=&\frac{\kappa+K}{\kappa H^2+(\kappa-K)^2}\Bigl((HG_H+2KG_K)(2(\kappa-K)+H^2)\nonumber\\
&&\hspace{4cm}-(2G_H+HG_K)H(\kappa+K)\Bigr)\nonumber\\
&=&\frac{\kappa+K}{\kappa H^2+(\kappa-K)^2}\Bigl(HG_H\bigl(2(\kappa-K)+H^2-2(\kappa+K)\bigr)\nonumber\\
&&+G_K\bigl(2K(2(\kappa-K)+H^2)-H^2(\kappa+K)\bigr)\Bigr)\nonumber\\
&=&\frac{(\kappa+K)(H^2-4K)}{\kappa H^2+(\kappa-K)^2}\Bigl(HG_H
-(\kappa-K)G_K\Bigr)\nonumber
\end{eqnarray}
So far, combining everything we have shown:
\begin{lemma}
Suppose $M$ is $2$-dimensional and $F:M\times[0,T)\to N$ evolves by (\ref{eq bas1}).
Then the evolution equation of a function $G$ that depends smoothly on the principal curvatures is given by
\begin{eqnarray}
\dt G-\uu \sigma ij\nabla_i\nabla_j G&=&\left(\uu Gij \sigma^{pq,kl}-\uu \sigma ij G^{pq,kl}\right)\nabla_i\dd hkl\nabla_j\dd hpq\label{eq gg7}\\
&&+\frac{(\kappa+K)(H^2-4K)}{\kappa H^2+(\kappa -K)^2}\Bigl(HG_H
-(\kappa-K)G_K\Bigr)\nonumber\\
&&-\phi\left((H^2+2(\kappa-K))G_H+H(\kappa+K)G_K\right)\,.\nonumber
\end{eqnarray}
\end{lemma}

The next lemma is interesting in its own right since it gives a precise relation between the full
norm of $\nabla h$ and $\nabla H$ on general $2$-dimensional Riemannian manifolds.
\begin{lemma}
Let $(M,g)$ be a $2$-dimensional Riemannian manifold and suppose $h\in\Gamma(T^*M\otimes T^*M)$ is a smooth symmetric Codazzi tensor,
i.e. in local coordinates we have $h=\dd hijdx^i\otimes dx^j$ and
\begin{eqnarray}
\dd hij&=&\dd hji\quad\forall\, i,j\,,\label{eq codsym}\\
\nabla_i\dd hjk&=&\nabla_j\dd hik\quad\forall\, i,j,k\,,\label{eq codcod}
\end{eqnarray}
where $\nabla$ denotes the Levi-Civita connection of $g$. Let $H=\uu gij\dd hij$ be the trace  and $\overset{\circ}{h}_{ij}:=\dd hij-H/2\,\dd gij$ be the tracefree part of $\dd hij$. Then the following identity holds:
\begin{eqnarray}
2|\overset{\circ}{h}|^2\bigl(|\nabla h|^2-|\nabla H|^2\bigr)=|\nabla|\overset{\circ}{h}|^2|^2-2\overset{\circ}{h}_{ij}\nabla^i|\overset{\circ}{h}|^2\nabla^jH\,.\label{eq kato}
\end{eqnarray}
\end{lemma}
\begin{proof}
Let $p\in M$ be arbitrary. At $p$ we choose normal coordinates such that $h$ is diagonal at $p$, say
$\dd hij=\operatorname{diag}(\lambda_1,\lambda_2)$. This is possible since $\dd hij$ is symmetric. 
Then we compute
\begin{eqnarray}
|\nabla\overset{\circ}{h}|^2
&=&\sum_{i,j,k=1}^2(\nabla_k\overset{\circ}{h}_{ij})^2\nonumber\\
&=&(\nabla_1\overset{\circ}{h}_{11})^2+2(\nabla_1\overset{\circ}{h}_{12})^2+(\nabla_2\overset{\circ}{h}_{11})^2\nonumber\\
&&+(\nabla_2\overset{\circ}{h}_{22})^2+2(\nabla_2\overset{\circ}{h}_{12})^2+(\nabla_1\overset{\circ}{h}_{22})^2\,.\nonumber
\end{eqnarray}
Since $\overset{\circ}{h}_{11}=-\overset{\circ}{h}_{22}$ we obtain $\nabla_i\overset{\circ}{h}_{11}=-\nabla_i\overset{\circ}{h}_{22}$ and then
\begin{eqnarray}
|\nabla\overset{\circ}{h}|^2
&=&4(\nabla_1\overset{\circ}{h}_{11})^2+4(\nabla_2\overset{\circ}{h}_{22})^2\nonumber\\
&&+2\left((\nabla_1\overset{\circ}{h}_{12})^2-(\nabla_2\overset{\circ}{h}_{11})^2+(\nabla_2\overset{\circ}{h}_{12})^2-(\nabla_1\overset{\circ}{h}_{22})^2\right)\,.\label{umb 2}
\end{eqnarray}
Next we compute
\begin{eqnarray}
|\nabla|\overset{\circ}{h}|^2|^2
&=&4\sum_{k=1}^2\left(\sum_{i,j=1}^2\overset{\circ}{h}_{ij}
\nabla_k\overset{\circ}{h}_{ij}\right)^2\nonumber\\
&=&4\left(
\overset{\circ}{h}_{11}\nabla_1\overset{\circ}{h}_{11}
+\overset{\circ}{h}_{22}\nabla_1\overset{\circ}{h}_{22}
\right)^2
+4\left(
\overset{\circ}{h}_{11}\nabla_2\overset{\circ}{h}_{11}
+\overset{\circ}{h}_{22}\nabla_2\overset{\circ}{h}_{22}
\right)^2\nonumber\\
&=&4\left(2\overset{\circ}{h}_{11}\nabla_1\overset{\circ}{h}_{11})\right)^2
+4\left(-2\overset{\circ}{h}_{11}\nabla_2\overset{\circ}{h}_{22})\right)^2\nonumber\\
&=&16(\overset{\circ}{h}_{11})^2\left((\nabla_1\overset{\circ}{h}_{11})^2
+(\nabla_2\overset{\circ}{h}_{22})^2\right)\,.\label{umb 3}
\end{eqnarray}
Combining (\ref{umb 2}), (\ref{umb 3}) and 
$|\overset{\circ}{h}|^2=(\overset{\circ}{h}_{11})^2+(\overset{\circ}{h}_{22})^2=2(\overset{\circ}{h}_{11})^2$
we get
\begin{equation}\label{umb 4}
2|\overset{\circ}{h}|^2\cdot|\nabla\overset{\circ}{h}|^2-|\nabla|\overset{\circ}{h}|^2|^2
=4|\overset{\circ}{h}|^2
\left((\nabla_1\overset{\circ}{h}_{12})^2-(\nabla_2\overset{\circ}{h}_{11})^2
+(\nabla_2\overset{\circ}{h}_{12})^2-(\nabla_1\overset{\circ}{h}_{22})^2\right).
\end{equation}
From $\nabla_i\dd hjk=\nabla_j\dd hik$ we obtain
$$\nabla_i\overset{\circ}{h}_{jk}-\nabla_j\overset{\circ}{h}_{ik}
=\frac{1}{2}\left(\nabla_jH\dd gik-\nabla_iH\dd gjk\right)$$
so that
$$(\nabla_1\overset{\circ}{h}_{12})^2-(\nabla_2\overset{\circ}{h}_{11})^2
=\frac{1}{2}(\nabla_1\overset{\circ}{h}_{12}+\nabla_2\overset{\circ}{h}_{11})\nabla_2H
=\nabla_2\overset{\circ}{h}_{11}\nabla_2H+\frac{1}{4}(\nabla_2H)^2$$
and
$$(\nabla_2\overset{\circ}{h}_{12})^2-(\nabla_1\overset{\circ}{h}_{22})^2
=\nabla_1\overset{\circ}{h}_{22}\nabla_1H+\frac{1}{4}(\nabla_1H)^2\,.$$
Then (\ref{umb 4}) implies
\begin{equation}\label{eq inter1}
2|\overset{\circ}{h}|^2\cdot|\nabla\overset{\circ}{h}|^2-|\nabla|\overset{\circ}{h}|^2|^2
=|\overset{\circ}{h}|^2
\left(|\nabla H|^2+4\nabla_1\overset{\circ}{h}_{22}\nabla_1H
-4\nabla_2\overset{\circ}{h}_{22}\nabla_2H\right)\,.
\end{equation}
In a next step we compute 
\begin{eqnarray}
\nabla_i|\overset{\circ}{h}|^2&=&2\overset{\circ}{h}_{11}\nabla_i\overset{\circ}{h}_{11}+2\overset{\circ}{h}_{22}\nabla_i\overset{\circ}{h}_{22}\nonumber\\
&=&-4\overset{\circ}{h}_{11}\nabla_i\overset{\circ}{h}_{22}\label{eq symsy}
\end{eqnarray}
where we have used that $|\overset{\circ}{h}|^2$ is a smooth function and $\overset{\circ}{h}_{ij}$ is diagonal and tracefree.
Then we get
\begin{eqnarray}
\overset{\circ}{h}_{ij}\nabla^i|\overset{\circ}{h}|^2\nabla^jH
&=&\overset{\circ}{h}_{11}\nabla_1|\overset{\circ}{h}|^2\nabla_1H+\overset{\circ}{h}_{22}\nabla_2|\overset{\circ}{h}|^2\nabla_2H\nonumber\\
&\overset{(\ref{eq symsy})}{=}&-4(\overset{\circ}{h}_{11})^2\nabla_1\overset{\circ}{h}_{22}\nabla_1H
-4\overset{\circ}{h}_{22}\overset{\circ}{h}_{11}\nabla_2\overset{\circ}{h}_{22}\nabla_2H\nonumber\\
&=&-(\overset{\circ}{h}_{11})^2(4\nabla_1\overset{\circ}{h}_{22}\nabla_1H
-4\nabla_2\overset{\circ}{h}_{22}\nabla_2H)\nonumber\\
&=&-\frac{1}{2}|\overset{\circ}{h}|^2(4\nabla_1\overset{\circ}{h}_{22}\nabla_1H
-4\nabla_2\overset{\circ}{h}_{22}\nabla_2H)\,.\nonumber
\end{eqnarray}
Combining with (\ref{eq inter1}) we get
\begin{eqnarray}
2|\overset{\circ}{h}|^2\cdot|\nabla\overset{\circ}{h}|^2-|\nabla|\overset{\circ}{h}|^2|^2
&=&|\overset{\circ}{h}|^2|\nabla H|^2-2\overset{\circ}{h}_{ij}\nabla^i|\overset{\circ}{h}|^2\nabla^jH\nonumber
\end{eqnarray}
and equation (\ref{eq kato}) follows from $|\nabla\overset{\circ}{h}|^2=|\nabla h|^2-\frac{1}{2}|\nabla H|^2$ .
\end{proof}

\begin{corollary}
\begin{enumerate}[i)]
\item From $\uu gkl\dd{\overset{\circ}{h}}ik\dd{\overset{\circ}{h}}jl=|\overset{\circ}{h}|^2\dd gij$ it follows
$$2|\overset{\circ}{h}|^2|\nabla\overset{\circ}{h}|^2=|\nabla_i|\overset{\circ}{h}|^2-\dd{\overset{\circ}{h}}ij\nabla^jH|^2$$
\item
In case $\nabla H=0$ we obtain the optimal Kato identity
$$(2|h|^2-H^2)|\nabla h|^2=|\nabla |h|^2|^2\,.$$
\item
Applying (\ref{eq kato}) to the mean curvature $H$ and the Gau\ss\ curvature $K=\det\ud hij$ we get
\begin{eqnarray}
&&(H^2-4K)\bigl(|\nabla h|^2-|\nabla H|^2\bigr)\nonumber\\
&&\hspace{2cm}=2(H\uu gij-\uu hij)\bigl(H\nabla_iH\nabla_jH-2\nabla_iH\nabla_jK\bigr)\nonumber\\
&&\hspace{2.4cm}- 2H\langle\nabla H,\nabla K\rangle+4|\nabla K|^2\,.\label{eq kato2}
\end{eqnarray}
\end{enumerate}
\end{corollary}

\subsection{$\bf C^2$-estimates}
Let us first compute the evolution equation of the Gau\ss\ curvature $K$. We use equation
(\ref{eq gg7}) with $G=K$. In this case we obtain
\begin{eqnarray}
G_K=1, \quad G_H=0, \quad \uu Gij=H\uu gij-\uu hij\,,\quad G^{pq,kl}=\uu gkl\uu gpq-\uu gpk\uu gql\,.\label{eq dergauss}
\end{eqnarray}
Then a straightforward computation using Codazzi's equation and equations (\ref{eq fpart4}), (\ref{eq kato2}) shows
\begin{eqnarray}
&&\left(\uu Gij \sigma^{pq,kl}-\uu \sigma ij G^{pq,kl}\right)\nabla_i\dd hkl\nabla_j\dd hpq\nonumber\\
&=&
\frac{1}{\bigl(\kappa H^2+(\kappa-K)^2\bigr)^2}\Bigl\{(\kappa^2-K^2)(H^2+2(\kappa-K))(|\nabla h|^2-|\nabla H|^2)\nonumber\\
&&\phantom{\frac{1}{\bigl(\kappa H^2+(\kappa-K)^2\bigr)^2}\Bigl\{}-(\kappa^2-K^2)H(H\uu gij-\uu hij)\nabla_iH\nabla_jH\nonumber\\
&&\phantom{\frac{1}{\bigl(\kappa H^2+(\kappa-K)^2\bigr)^2}\Bigl\{}-2\kappa H^2
(H\uu gij-\uu hij)\nabla_iH\nabla_jK\nonumber\\
&&\phantom{\frac{1}{\bigl(\kappa H^2+(\kappa-K)^2\bigr)^2}\Bigl\{}+2H(\kappa-K)
(H\uu gij-\uu hij)\nabla_iK\nabla_jK\nonumber\\
&&\phantom{\frac{1}{\bigl(\kappa H^2+(\kappa-K)^2\bigr)^2}\Bigl\{}
-(\kappa-K)^2\bigl(H\langle\nabla H,\nabla K\rangle-2|\nabla K|^2\bigr)\Bigr\}\,.\nonumber
\end{eqnarray}
Inserting this into equation (\ref{eq gg7}) we have shown
\begin{lemma}
In dimension $2$ there exists a smooth function $S$ and a smooth vector field $V$ such that the
evolution equation of the Gau\ss\ curvature $K$ induced by the flow (\ref{eq bas1}) can be
written in the form
\begin{eqnarray}
\dt K=\uu\sigma ij\nabla_i\nabla_j K+\langle\nabla K,V\rangle+(\kappa^2-K^2)S-(\kappa+K)H\phi\,.\label{evol kshort}
\end{eqnarray}
\end{lemma}
\begin{lemma}\label{lemm kest}
Suppose $n=2$ and that $K^2<\kappa^2$ at $t=0$. Then there exists a constant $\varepsilon>0$
such that  the estimates
\begin{eqnarray}
\kappa-K&>&\varepsilon\,,\label{est 1}\\
\kappa+K&>& 0\label{est 2}
\end{eqnarray}
hold for all $t\in[0,T)$.
\end{lemma}
\begin{proof}
For real numbers $x,y$ with $xy<1$ one has
$$\arctan x+\arctan y=\arctan\frac{x+y}{1-xy}\,.$$
If $\lambda_1,\lambda_2$ denote the two principal curvatures, then $\kappa-K>0$ is equivalent to
$$\frac{\lambda_1}{\sqrt{\kappa}}\cdot\frac{\lambda_2}{\sqrt{\kappa}}<1$$ and hence we
have
$$\phi=\frac{1}{\sqrt{\kappa}}\left(\arctan\frac{\lambda_1}{\sqrt{\kappa}}+\arctan\frac{\lambda_2}{\sqrt{\kappa}}\right)=\frac{1}{\sqrt{\kappa}}\arctan\frac{\sqrt{\kappa}\, H}{\kappa-K}\,.$$
Lemma \ref{est phi} implies that there exists a positive constant $C<\pi/2$, depending only on 
$\sup_{p\in M}|\phi(p,0)|$, such that
$$\kappa H^2\le(\kappa-K)^2 \tan^2(Ce^{-2t})$$
as long as $\kappa-K\ge 0$. Adding $(\kappa-K)^2$ to both sides we get
$$\kappa^2\le(\kappa+\lambda_1^2)(\kappa+\lambda_2^2)=\kappa H^2+(\kappa-K)^2\le(\kappa-K)^2\bigl(1+\tan^2(Ce^{-2t})\bigr)$$
and therefore
$$(\kappa-K)^2\ge\kappa^2\cos^2(Ce^{-2t})>0$$
as long as $\kappa-K\ge0$. This shows that $\kappa-K$ cannot tend to zero and that inequality
(\ref{est 1}) holds with $\varepsilon:=\kappa\cos(C)$.
The second inequality (\ref{est 2}) follows from
the evolution equation of $K$ given by equation (\ref{evol kshort}) and the maximum principle.
\end{proof}
So far we have shown that the Gau\ss\ curvature $K$ stays uniformly bounded. To prevent the
formation of thin "necks" we need to control the full norm $|h|$ of the second fundamental form.
But this can be achieved by exploiting  the bounds of $K$ and $\tan(\sqrt{\kappa}\phi)$ since
$$\kappa|h|^2=\kappa H^2-2\kappa K=(\kappa-K)^2\tan^2(\sqrt{\kappa}\phi)-2\kappa K$$
and the RHS is bounded.

We summarize:
\begin{lemma}
Suppose $n=2$ and that the Gau\ss\ curvature satisfies $|K|<\kappa$ at $t=0$. Then the second fundamental form $h$ stays uniformly bounded for all $t\in[0,T)$.
\end{lemma}

\subsection{$\bf C^1$-estimates}
Once we have uniform $C^2$-estimates it is easy to derive uniform $C^1$-estimates. Since the 
induced Riemannian metric $\dd gij$ evolves according to 
$$\dt\dd gij=2\phi\dd hij$$
and the second fundamental tensor $\dd hij$ is uniformly bounded, we conclude that there exists
a uniform constant $C'>0$ such that
$$\left|\dt\dd gij\right|\le C'|\phi|\dd gij$$
holds for all $t\in[0,T)$. By Lemma \ref{est phi} this can be further estimated and we obtain
$$\left|\dt\dd gij\right|\le Ce^{-t}\dd gij$$
with another uniform constant $C>0$. Thus
$$e^{-C(1-e^{-t})}\dd gij(p,0)\le \dd gij(p,t)\le e^{C(1-e^{-t})}\dd gij(p,0)$$
for all $(p,t)\in M\times [0,T)$. In particular, all induced metrics are uniformly equivalent to the
initial metric.

\subsection{Longtime existence and convergence}
Using the key estimates obtained in the previous subsections, we are now able to
prove Theorem \ref{theo main2}.

{\bf Proof of Theorem \ref{theo main2}:}
Since we have uniform $C^2$-{estimates we  proceed as in \cite{Andrews} (see also \cite{Andrews2} for details) 
and use a parabolic variant of the classical Morrey and Nirenberg estimates for fully nonlinear elliptic
equations in two space variables to obtain uniform
$C^{2,\alpha}$-estimates in space and $C^{1,\alpha}$-estimates in time. Schauder theory then implies uniform $C^k$-estimates in space  and time for any $k\ge 1$ and hence longtime existence of a smooth solution.  Since in local coordinates we have $\dt F^\alpha=\phi\nu^\alpha$ we compute for
any $0\le t_1\le t_2<T$
\begin{eqnarray}
\left|F^\alpha(p,t_2)-F^\alpha(p,t_1)\right|
&=&\left|\int\limits_{t_1}^{t_2}\phi(p,t)\nu^\alpha(p,t)dt\right|\nonumber\\
&\le&C(e^{-2t_1}-e^{-2t_2})\label{eq expon}
\end{eqnarray}
with some constant $C>0$ independent of $t_1,t_2$, where we have used Lemma \ref{est phi} and the fact that all
induced Riemannian metrics are uniformly equivalent. 
Thus for any $\epsilon>0$ and any $p\in M$ there exists some 
$t_1>0$ such that for all $t_2\ge t_1$ the points $F(p,t_2)$ and $F(p,t_1)$ lie in the same coordinate chart and
the euclidean distance $|F(p,t_2)-F(p,t_1)|$ in this coordinate chart
is bounded by $\epsilon$. Since $M$ is compact we obtain uniform and by (\ref{eq expon}) exponential convergence
of $M_t=F(M,t)$ to a smooth limiting surface $M_\infty\subset N$ as $t\to\infty$. $M_\infty$ is spacelike since all
induced Riemannian metrics stay uniformly equivalent.  Since by Lemma \ref{lemm kest} the estimate $\kappa-K>\epsilon$ holds 
for some $\epsilon >0$ and all $t\ge 0$ we can express $\phi$ in the form
 $\phi=\frac{1}{\sqrt{\kappa}}\arctan\frac{\sqrt{\kappa}H}{\kappa-K}$ 
for all $t\in[0,\infty)$. Then Lemma \ref{est phi}  implies that $\phi$ and $H$ tend to zero as $t\to\infty$ and hence
the mean curvature $H$ of the limiting surface vanishes. Again by Lemma \ref{lemm kest} we obtain that the Gau\ss\ curvature of the limiting surface $M_\infty$ satisfies $-\kappa\le K\le\kappa-\epsilon<\kappa$. This proves Theorem \ref{theo main2}.
\hfill$\square$

\section{Appendix}\label{sec 5}
With the same notations as before let us now assume that $F:M\to N$ is a spacelike immersion of
a surface $M$ into the anti-De Sitter manifold $N=\ads$,  represented by
$$\ads=\{V\in\mathbb{R}_2^4:\sca VV=-1\}$$
with its induced Lorentzian metric $\scb\cdot\cdot$ that it inherits from
$\mathbb{R}_2^4=\bigl(\real{4},\sca\cdot\cdot\bigr)$, where the inner product on $\mathbb{R}_2^4$
is given by
$$\sca VW=V^1W^1+V^2W^2-V^3W^3-V^4W^4\,.$$
$\ads$ is a $3$-dimensional Lorentzian space form of constant sectional
curvature $-1$ and is a vacuum solution of Einstein's equation with cosmological
constant $\Lambda<0$. $\ads$ contains closed timelike
curves and hence is not simply connected. The simply connected universal
cover of $\ads$ will also be called anti-De Sitter space and we denote it  $\widetilde{\ads}$.

Thus an immersion $F:M\to\ads$ can also be seen as an immersion of $M$ into $\mathbb{R}_2^4$.
The Gau\ss\ map of $F$ is the map
$$\mathscr{G}:M\to Gr_2^+(2,4)\,,\quad p\mapsto T_pM\in Gr_2^+(2,4)\,,$$
where $T_pM$ is considered as an oriented spacelike surface in $\mathbb{R}_2^4$
and $Gr_2^+(2,4)$ denotes the Grassmannian of oriented spacelike surfaces in $\mathbb{R}_2^4$.
It is well known that  the Grassmannian $Gr_2^+(2,4)$ is isometric to $\mathbb{H}_{1/\sqrt{2}}\times\mathbb{H}_{1/\sqrt{2}}$, where $\mathbb{H}_{1/\sqrt{2}}$
denotes the scaled hyperbolic plane
$$\mathbb{H}_{1/\sqrt{2}}:=\left\{V\in\mathbb{R}_1^3:\scc VV=-\frac{1}{2}\,, V^3>0\right\}\,,$$
where $\bigl(\mathbb{R}_1^3,\scc\cdot\cdot\bigr)$ denotes the usual Minkowski space.
To understand the geometry of the Gau\ss\ map $\mathscr{G}$ it is convenient to use the
isometry between $Gr_2^+(2,4)$ and $\mathbb{H}_{1/\sqrt{2}}\times\mathbb{H}_{1/\sqrt{2}}$
(since the sectional curvature of $\mathbb{H}_{1/\sqrt{2}}$ is $-2$, $Gr_2^+(2,4)$ is a K\"ahler-Einstein
manifold of scalar curvature $S=-8$).

Let $e_1,e_2,e_3,e_4$ denote the standard basis of $\real{4}$.
We introduce a set of endomorphisms on $\real{4}$:
\begin{eqnarray}
E_+^1:\,
\begin{pmatrix} e_1\\e_2\\e_3\\e_4\end{pmatrix}
\mapsto
\begin{pmatrix} e_4\\-e_3\\-e_2\\e_1\end{pmatrix},&\quad&E_-^1:\,
\begin{pmatrix} e_1\\e_2\\e_3\\e_4\end{pmatrix}
\mapsto
\begin{pmatrix} -e_4\\-e_3\\-e_2\\-e_1\end{pmatrix}
,\nonumber\\[10pt]
E_+^2:\,
\begin{pmatrix} e_1\\e_2\\e_3\\e_4\end{pmatrix}
\mapsto
\begin{pmatrix} -e_3\\-e_4\\-e_1\\-e_2\end{pmatrix},&\quad&E_-^2:\,
\begin{pmatrix} e_1\\e_2\\e_3\\e_4\end{pmatrix}
\mapsto
\begin{pmatrix} -e_3\\e_4\\-e_1\\e_2\end{pmatrix}
,\nonumber\\[10pt]
E_+^3:\,
\begin{pmatrix} e_1\\e_2\\e_3\\e_4\end{pmatrix}
\mapsto
\begin{pmatrix} e_2\\-e_1\\-e_4\\e_3\end{pmatrix},&\quad&E_-^3:\,
\begin{pmatrix} e_1\\e_2\\e_3\\e_4\end{pmatrix}
\mapsto
\begin{pmatrix} e_2\\-e_1\\e_4\\-e_3\end{pmatrix}
,\nonumber\\
\end{eqnarray}
These endomorphisms satisfy
$$E_+^1E_+^2=-E^2_+E^1_+=E^3_+\,,\quad (E^1_+)^2=(E^2_+)^2=-(E^3_+)^2=\operatorname{Id}\,,$$
$$E_-^1E_-^2=-E^2_-E^1_-=E^3_-\,,\quad (E^1_-)^2=(E^2_-)^2=-(E^3_-)^2=\operatorname{Id}$$
Moreover, if $V\in\real{4}$ is an arbitrary nonzero vector, then
$$\{V,E_+^1V,E_+^2V,E_+^3V\}$$
forms a positively oriented basis and
$$\{V,E_-^1V,E_-^2V,E_-^3V\}$$
a negatively oriented basis of $\real{4}$.
Associated to these endomorphisms are the following six symplectic $2$-forms:
$$\omega^A_+:=\sca {E^A_+\cdot}\cdot\quad\text{and}\quad\omega^A_-:=\sca{E^A_-\cdot}\cdot\,,\quad A=1,2,3$$
and we have
\begin{eqnarray}
&\omega^1_+=e_2\wedge e_3+e_4\wedge e_1,\quad&\omega^1_-=e_2\wedge e_3-e_4\wedge e_1,\nonumber\\
&\omega^2_+=e_1\wedge e_3+e_2\wedge e_4,\quad&\omega^2_-=e_1\wedge e_3-e_2\wedge e_4,\nonumber\\
&\omega^3_+=e_1\wedge e_2+e_3\wedge e_4,\quad&\omega^3_-=e_1\wedge e_2-e_3\wedge e_4.\nonumber
\end{eqnarray}
For any spacelike unit vector $e$ and any two vectors $V,W$ we have
\begin{eqnarray}
\sca VW&=&-\omega^1_+(e,V)\omega^1_+(e,W)-\omega^2_+(e,V)\omega^2_+(e,W)\nonumber\\
&&+\omega^3_+(e,V)\omega^3_+(e,W)+\sca eV\sca eW\nonumber\\
&=&-\omega^1_-(e,V)\omega^1_-(e,W)-\omega^2_-(e,V)\omega^2_-(e,W)\nonumber\\
&&+\omega^3_-(e,V)\omega^3_-(e,W)+\sca eV\sca eW.\label{eq decomp}
\end{eqnarray}
Let $\nu_N$ denote the future directed timelike unit normal along $\ads\subset\mathbb{R}_2^4$.
The orientation of $\real{4}$ induces an orientation on $\ads$ in the following way: We say that $V_1,V_2,V_3\in  T_q\hspace{-2pt}\ads$ is positively oriented, if $V_1,V_2,V_3,\nu_N$ represents the positive orientation of $\real{4}$.

Now let $F:M\to\ads$ be a spacelike immersion of an oriented surface and let $\nu$ denote the future
directed timelike unit normal of $M$ within $\ads$. We will assume that the orientation of $M$ is 
chosen in such a way that for any positively oriented basis $\{e_1,e_2\}$ of $T_pM$ the
basis $\{DF(e_1),DF(e_2),\nu\}$ represents the positive orientation of $T_{F(p)}\hspace{-2pt}\ads$.

For such an immersion let us define the six functions
$$\mathscr{G}_+^A:M\to\real{}\,,\quad \mathscr{G}_+^A:=\frac{1}{\sqrt{2}}\hodge\left(F^*\omega_+^A\right)\,,\quad A=1,2,3$$
and
$$\mathscr{G}_-^A:M\to\real{}\,,\quad \mathscr{G}_-^A:=\frac{1}{\sqrt{2}}\hodge\left(F^*\omega_-^A\right)\,,\quad A=1,2,3\,,$$
where $F^*$ denotes "pull-back" and $\ast$ is the Hodge-Operator on forms.
From equation (\ref{eq decomp}) one immediately gets \footnote{Choose an arbitrary unit tangent vector $e$ of $T_pM$ and let $V=W=Ce$, where $C$ denotes the complex structure on $M$ induced by
the orientation and Riemannian metric on $M$}
\begin{equation}\label{eq gaussmap}
(\mathscr{G}_+^1)^2+(\mathscr{G}_+^2)^2-(\mathscr{G}_+^3)^2=-\frac{1}{2}=(\mathscr{G}_-^1)^2+(\mathscr{G}_-^2)^2-(\mathscr{G}_-^3)^2
\end{equation}
and by construction we have $\mathscr{G}^3_\pm>0$,  so that
$$\mathscr{G}_+=(\mathscr{G}_+^1,\mathscr{G}_+^2,\mathscr{G}_+^3)\quad\text{and}\quad
\mathscr{G}_-=(\mathscr{G}_-^1,\mathscr{G}_-^2,\mathscr{G}_-^3)$$
define two functions from $M$ to $\mathbb{H}_{1/\sqrt{2}}$. $\mathscr{G}_+,\mathscr{G}_-$
are called the self-dual resp. the anti-self-dual Gau\ss\ maps of $F$ and the Gau\ss\ map
$\mathscr{G}:M\to Gr^+_2(2,4)=\mathbb{H}_{1/\sqrt{2}}\times\mathbb{H}_{1/\sqrt{2}}$
is given by the pair $\mathscr{G}=(\mathscr{G}_+,\mathscr{G}_-)$.

Let 
$$\langle\cdot,\cdot\rangle_{~_{\hspace{-6pt} Gr_2^+(2,4)}}\,,\quad\mathscr{J}\quad\text{and}\quad
\omega=\langle\mathscr{J}\cdot,\cdot\rangle_{~_{\hspace{-6pt} Gr_2^+(2,4)}}$$ 
denote the K\"ahler metric, complex structure and K\"ahler form on the Grassmannian $Gr_2^+(2,4)=\mathbb{H}_{1/\sqrt{2}}\times\mathbb{H}_{1/\sqrt{2}}$.

As was shown in \cite{Torralbo}, \cite{Torralbo-Urbano}, we have
$\mathscr{G}^*\omega=0$, i.e. the Gau\ss\ map of an immersion $F:M\to\ads$ defines a Lagrangian immersion $\mathscr{G}:M\to Gr_2^+(2,4)$.

Let
$$\dd \sigma ij dx^i\otimes dx^j=\mathscr{G}^*\langle\cdot,\cdot\rangle_{~_{\hspace{-6pt} Gr_2^+(2,4)}}$$
denote the Riemannian metric on $M$ induced by the Gau\ss\ map. If $D$ denotes the connection
associated to $\sigma$, then it is well known that the second fundamental tensor
$$\ddd\tau ijk=\omega (D_i\mathscr{G},D_jD_k\mathscr{G})$$
of the Lagrangian immersion is completely symmetric and that the mean curvature form
$\tau=\tau_idx^i$ on $M$, i.e. its trace $\tau_i=\uu\sigma jk\ddd\tau ijk$ (where
$(\uu\sigma jk)_{j,k=1,2}$ denotes the inverse of $(\dd\sigma jk)_{j,k=1,2}$), is closed.

The first and second fundamental forms on $M$ induced by $F$
shall be denoted (as before) by $\dd gij dx^i\otimes dx^j$ and $\dd hij dx^i\otimes dx^j$.
If we consider $F$ as a map from $M$ to $\mathbb{R}_2^4$, then the Gau\ss\ formula shows
that the second fundamental tensor $\tilde A$ of $M$, considered as a submanifold of codimension two
in $\mathbb{R}_2^4$, decomposes into
\begin{equation}\label{eq gaussmapdec}
\dd{\tilde A}ij=\dd gij\nu_N+\dd hij\nu\,.
\end{equation}
\begin{lemma}
Let $F:M\to\ads$ be a spacelike immersion. With the same notations as above 
the following relations between the first and second fundamental forms of $F$ and $\mathscr{G}$
are valid:
\begin{eqnarray}
\dd\sigma ij&=&\dd gij+\uu gkl\dd hik\dd hjl\,,\label{eq rela1}\\
\ddd\tau ijk&=&\nabla_i\dd hjk\,,\label{eq rela2}
\end{eqnarray}
where $\nabla$ denotes the Levi-Civita connection of $\dd gij$.
\end{lemma}
\begin{proof}
Straightforward computations using (\ref{eq decomp}), (\ref{eq gaussmap}) and (\ref{eq gaussmapdec}).
\end{proof}
In particular, we observe that $\dd\sigma ij$ coincides with the tensor defined earlier in equation
(\ref{defi sigma}) since in this special situation we have $\kappa=1$.
As a corollary we obtain:
\begin{lemma}\label{lemm app1}
The Maslov class of the Gau\ss\ map $\mathscr{G}$ is trivial and the Lagrangian angle is
given by $\phi=\arctan\lambda_1+\arctan\lambda_2$, where $\lambda_1,\lambda_2$ are the 
principal curvatures of $F:M\to\ads$.
\end{lemma}
\begin{proof}
We have seen in Lemma \ref{lemm anglesmooth} that $\phi=\arctan\lambda_1+\arctan\lambda_2$
is a smooth function. Moreover we have $\frac{\partial\phi}{\partial\dd hij}=\uu\sigma ij$ and then
\begin{eqnarray}
\nabla_k\phi&=&\frac{\partial\phi}{\partial\dd hij}\nabla_k\dd hij+\frac{\partial\phi}{\partial\uu gij}\nabla_k\uu gij\nonumber\\
&=&\uu\sigma ij\nabla_k\dd hij\nonumber\\
&\overset{(\ref{eq rela2})}{=}&\uu\sigma ij\ddd\tau kij\nonumber\\
&=&\tau_k\,.\nonumber
\end{eqnarray}
This means that the mean curvature form $\tau$ of the Gau\ss\ map satisfies $\tau=d\phi$.
Since $\tau/\pi$ represents the Maslov class, it must be trivial.
\end{proof}

We will now treat the case where $F:M\times[0,T)\to\ads$ is a smooth family of
spacelike immersions satisfying an evolution equation of the form
$$\dt F=f\nu\,,$$
where $f$ is an arbitrary smooth function and $\nu$ the future directed timelike unit normal.
The Gau\ss\ maps of $F$ depend on $t$ and will vary in time. A straightforward computation
gives the two relations
$$
\left\langle \dt\mathscr{G},D_k\mathscr{G}\right\rangle_{~_{\hspace{-6pt} Gr_2^+(2,4)}}
=\uu gml\dd hlk\nabla_mf$$
and
$$
\left\langle \dt\mathscr{G},\mathscr{J}D_k\mathscr{G}\right\rangle_{~_{\hspace{-6pt} Gr_2^+(2,4)}}
=\nabla_kf\,,$$
so that
\begin{equation}\nonumber
\dt\mathscr{G}=\mathscr{J}\left(\uu\sigma kl\nabla_kfD_l\mathscr{G}\right)+\uu\sigma kl\uu gms\dd hskD_mfD_l\mathscr{G}\,.
\end{equation}
So we have shown:
\begin{lemma}\label{lemm app2}
Suppose $F:M\times[0,T)\to\ads$ is a smooth family of spacelike immersions driven by the flow
$\dt F=f\nu$, where $\nu$ denotes the future directed timelike unit normal. Then the Gau\ss\
maps $\mathscr{G}_F:M\times[0,T)\to Gr_2^+(2,4)$ of $F$ evolve according to
\begin{equation}\label{eq evolgauss}
\dt\mathscr{G}_F=\left(\mathscr{J}\circ d\mathscr{G}_F+d\mathscr{G}_F\circ\mathscr{W}_F\right)\nabla^\sigma f\,.
\end{equation}
where$\mathscr{J}$ denotes the complex structure on $Gr_2^+(2,4)$, 
$\mathscr{W}_F$ is the Weingarten map of $F$
and $\nabla^\sigma f$ denotes the gradient of $f$ w.r.t. the
induced metric $\sigma=\mathscr{G}^*\left\langle \cdot,\cdot\right\rangle_{~_{\hspace{-6pt} Gr_2^+(2,4)}}$.
\end{lemma}

In particular, if we choose for $f$ the Lagrangian angle $\phi=\arctan\lambda_1+\arctan\lambda_2$,
then - up to the tangential term $(d\mathscr{G}_F\circ\mathscr{W}_F)\nabla^\sigma f$, which is of no interest concerning the geometric
evolution -
the Gau\ss\ maps evolve by the Lagrangian mean curvature flow.

\end{document}